\documentclass[12pt,a4paper]{amsart}
\usepackage{a4wide}
\usepackage{amsfonts,amsthm,amsmath,amssymb,amscd}
\usepackage{graphicx,color}

\theoremstyle{plain}
\newtheorem{lem}{Lemma}
\newtheorem{prop}[lem]{Proposition}
\newtheorem{thm}[lem]{Theorem}

\newtheorem*{thmA}{Theorem A}
\newtheorem*{corA'}{Corollary A'}
\newtheorem*{thmB}{Theorem B}
\newtheorem*{thmC}{Theorem C}

\newtheorem*{DWthm}{Denjoy--Wolff Theorem}
\newtheorem*{Cowenthm}{Cowen's Theorem}
\newtheorem*{Schwarzlem}{Schwarz--Pick's Lemma}

\theoremstyle{definition}
\newtheorem{defn}[lem]{Definition}
\newtheorem*{defn*}{Definition}

\newtheorem*{ex*}{Example}
\newtheorem{rem}[lem]{Remark}
\newtheorem*{rem*}{Remark}

\theoremstyle{remark}

\DeclareMathOperator{\dist}{dist}
\DeclareMathOperator{\Arg}{Arg}
\DeclareMathOperator{\ext}{ext}

\newcommand{\C}{\mathbb C}
\newcommand{\D}{\mathbb D}
\newcommand{\HH}{\mathbb H}
\newcommand{\clC}{\overline \C}

\newcommand{\Z}{\mathbb Z}
\newcommand{\N}{\mathbb N}

\newcommand{\DD}{\mathcal D}
\newcommand{\PP}{\mathcal P}

\newcommand{\bd}{\partial}
\renewcommand{\Re}{\textup{Re}}
\renewcommand{\Im}{\textup{Im}}

\begin{document}

\title{Absorbing sets and Baker domains for holomorphic maps}

\date{}

\author{Krzysztof Bara\'nski}
\address{Institute of Mathematics, University of Warsaw,
ul.~Banacha~2, 02-097 Warszawa, Poland}
\email{baranski@mimuw.edu.pl}

\author{N\'uria Fagella}
\address{Departament de Matem\`atica Aplicada i An\`alisi,
Universitat de Barcelona, 08007 Barce\-lona, Spain}
\email{fagella@maia.ub.es}

\author{Xavier Jarque}
\address{Departament de Matem\`atica Aplicada i An\`alisi,
Universitat de Barcelona, 08007 Barce\-lona, Spain}
\email{xavier.jarque@ub.edu}

\author{Bogus{\l}awa Karpi\'nska}
\address{Faculty of Mathematics and Information Science, Warsaw
University of Technology, ul.~Koszykowa~75, 00-662 Warszawa, Poland}
\email{bkarpin@mini.pw.edu.pl}

\thanks{Supported by Polish NCN grant decision DEC-2012/06/M/ST1/00168. The second and third authors were partially supported by the Catalan grant 2009SGR-792, and by the Spanish grant MTM2011-26995-C02-02.}

\bibliographystyle{plain}

\begin{abstract} We consider holomorphic maps $f: U \to U$ for a hyperbolic domain $U$ in the complex plane, such that the iterates of $f$ converge to a boundary point $\zeta$ of $U$. By a previous result of the authors, for such maps there exist nice absorbing domains $W \subset U$. In this paper we show that $W$ can be chosen to be simply connected, if $f$ has parabolic~I type in the sense of the Baker--Pommerenke--Cowen classification of its lift by a universal covering (and $\zeta$ is not an isolated boundary point of $U$). Moreover, we provide counterexamples for other types of the map $f$ and give an exact characterization of parabolic~I type in terms of the dynamical behaviour of $f$.
\end{abstract}

\maketitle

\section{Introduction}\label{sec:intro}

In this paper we study iterates $f^n = \underbrace{f\circ \cdots \circ f}_{n \textrm{ times}}$ of a holomorphic map
\[
f: U \to U,
\]
where $U$ is a hyperbolic domain in the complex plane $\C$ (i.e.~a domain whose complement in $\C$ contains at least two points) and $f$ has no fixed points, that is $f(z) \neq z$ for $z \in U$. In the special case when $U$ is the unit disc $\D$ or the right half-plane $\HH$, the dynamical behaviour of $f$ has been extensively studied, starting from the works of Denjoy, Wolff and Valiron in the 1920's and 1930's \cite{denjoy,valiron,valiron-book,wolff}. In particular, the celebrated Denjoy--Wolff Theorem asserts that the iterates of $f$ converge almost uniformly (i.e. uniformly on compact subsets of $\D$) as $n \to \infty$ to a point $\zeta$ in the boundary of $U$. Changing the coordinates by a M\"obius map, we can conveniently assume in this case $U = \HH$, $\zeta = \infty$. Some years later, Baker and Pommerenke \cite{baker-pommer,pommer} and Cowen \cite{cowen} proved that $f$ on $\HH$ is semi-conjugated to a M\"obius map $T: \Omega \to \Omega$ by a holomorphic map $\varphi: \HH \to \Omega$, where the 
following three cases can occur: 
\begin{itemize}
\item $\Omega = \C$, $T(\omega) = \omega + 1$ $($parabolic~I type$)$,
\item $\Omega = \HH$, $T(\omega) = \omega \pm i$ $($parabolic~II type$)$,
\item $\Omega = \HH$, $T(\omega) = a\omega$ for some $a > 1$ $($hyperbolic type$)$
\end{itemize}
(see Section~\ref{sec:back} for a precise formulation). 

For an arbitrary hyperbolic domain $U \subset \C$, the problem of describing the dynamics of a holomorphic map $f: U \to U$ without fixed points is more complicated. To this aim, one can consider a lift $g: \HH \to \HH$ of $f$ by a universal covering $\pi: \HH \to U$. Some results on the dynamics of $f$ were obtained by Marden and Pommerenke \cite{marden-pommer} and Bonfert \cite{bonfert}, who proved that if $f$ {\em has no isolated boundary fixed points} (i.e. points $\zeta$ in the boundary of $U$ in $\clC$ such that $f$ extends holomorphically to $U \cup \{\zeta\}$ as $f(\zeta) = \zeta$, see Definition \ref{ibfp}),  then it is semi-conjugated to a M\"obius map on $\C$ or on a hyperbolic domain in $\C$. In 1999, K\"onig \cite{konig} extended the Baker--Pommerenke--Cowen result on the semiconjugacy of $f$ to a M\"obius map for the case, when $f^n \to \infty$ as $n \to \infty$ and every closed loop in $U$ is eventually contractible in $U$ under iteration of $f$. 

One can extend the classification of $f$ into the three types (parabolic~I, parabolic~II and hyperbolic), defining its type by the type of its lift $g$ (see Section~\ref{sec:parab}).
In \cite{konig}, K\"onig characterized the three types of $f$ (under the restriction on eventual contractibility of loops in $U$) in terms of the behaviour of the sequence $|f^{n+1}(z) - f^n(z)|/\dist(f^n(z),\bd U)$ for $z \in U$ (see Theorem~\ref{thm:konig}), where  $\bd U$ denotes the boundary of $U$ in $\C$ and $\dist (f^n(z),\bd U) = \inf_{u \in \bd U} |f^n(z)-u|$.

In this paper we present a characterization of maps $f$ of parabolic~I type in terms of their dynamical properties in the general case, where $f$ is an arbitrary holomorphic map without fixed points on a hyperbolic domain $U\subset \C$.  More precisely, we prove the following.
\begin{thmA}
Let $U$ be a hyperbolic domain in $\C$ and let $f: U \to U$ be a holomorphic
map without fixed points and without isolated boundary fixed points. Then the
following statements are equivalent:
\begin{itemize}
\item[$(a)$] $f$ is of parabolic~I type,
\item[$(b)$] $\varrho_U(f^{n+1}(z), f^n(z)) \to
0$ as $n \to \infty$ for some $z \in U$, 
\item[$(c)$] $\varrho_U(f^{n+1}(z), f^n(z)) \to
0$ as $n \to \infty$ almost uniformly on $U$, 
\item[$(d)$] $|f^{n+1}(z)- f^n(z)|/\dist(f^n(z), \bd U) \to
0$ as $n \to \infty$ for some $z \in U$,
\item[$(e)$] $|f^{n+1}(z)- f^n(z)|/\dist(f^n(z), \bd U) \to
0$ as $n \to \infty$ almost uniformly on $U$.
\end{itemize}
\end{thmA}

For the other two types of $f$ we prove that if $\inf_{z\in U} \lim_{n\to\infty}\varrho_U(f^{n+1}(z), f^n(z)) > 0$, then $f$ is of hyperbolic type (see Proposition~\ref{prop:hyper} and Remark~\ref{rem:konig}).

Another question we consider in this paper is the existence and properties of \emph{absorbing domains} in $U$ for $f$.

\begin{defn*}[{\bf Absorbing domain}{}] Let $U$ be a domain in $\C$ and let $f: U \to U$ be a holomorphic map. A domain $W \subset U$ is called \emph{absorbing} in $U$ for $f$, if $f(W) \subset W$ and for every compact set $K \subset U$ there exists $n \ge 0$, such that $f^n(K) \subset W$. 
\end{defn*}

The problem of existence of suitable absorbing domains for holomorphic maps has a long history and is related to the study of local behaviour of a holomorphic map near a fixed point and properties of the Fatou components in the theory of the dynamics of rational, entire and meromorphic maps. (For basic facts about the dynamics of holomorphic maps we refer to \cite{bergweiler,carlesongamelin}.) For instance, if $U$ is a neighbourhood of an attracting fixed point $\zeta$ of $f$ (e.g.~if $U$ is the immediate basin of an attracting periodic point $\zeta$ of period $p$ of a meromorphic map $\tilde f$ and $f = \tilde f^p$), then $f$ is conformally conjugate by a map $\phi$ to $w\mapsto f'(\zeta) w$ (if $0 < |f'(\zeta)| < 1$) or $w\mapsto w^k$ for some integer $k > 1$ (if $f'(\zeta)= 0$) near $w=0$, and $W = \phi^{-1}(\D(0, \varepsilon))$ for a small $\varepsilon > 0$ is a simply connected absorbing domain in $U$ for $f$, such that $f(\overline{W}) \subset W$ and $\bigcap_{n\geq 0} f^n(\overline{W}) = \{\zeta\}$
(see e.g.~\cite{carlesongamelin}). 

From now on, assume that 
\[
f: U \to U
\]
is a holomorphic map on a hyperbolic domain $U \subset \C$ and the iterates of $f$ converge to a boundary point $\zeta$ of $U$. Changing the coordinates by a M\"obius map, we can assume $\zeta = \infty$, so 
\[
f^n \to \infty \quad \text{as} \quad  n \to \infty
\]
almost uniformly on $U$. Since the above definition of an absorbing domain is quite wide, (observe for instance that the whole domain $U$ is always absorbing for $f$), we introduce a notion of a \emph{nice absorbing domain}.

\begin{defn*}[{\bf Nice absorbing domain}{}] An absorbing domain $W$ in a domain $U \subset \C$ for a holomorphic map $f: U \to U$ with $f^n \to \infty$ is called \emph{nice}, if
\begin{itemize}
\item[$(a)$] $\overline{W} \subset U$,
\item[$(b)$] $f^n(\overline{W}) = \overline{f^n(W)} \subset f^{n-1}(W)$ for every $n\geq 1$,
\item[$(c)$] $\bigcap_{n=1}^\infty f^n(\overline{W}) = \emptyset$.
\end{itemize}
\end{defn*}

An example of a nice absorbing domain is an attracting petal $W$ in a basin $U$ of a parabolic $p$-periodic point $\zeta = \infty$ for a rational map $\tilde f$, where $f = \tilde f^p$ (see e.g.~\cite{carlesongamelin}). 

The question of existence of absorbing regions in hyperbolic domains $U$ is particularly interesting in studying the dynamics of entire and meromorphic maps with Baker domains. Recall that a $p$-periodic Baker domain for a transcendental meromorphic map $\tilde f: \C \to \clC$ is a Fatou component $U \subset \C$, such that $\tilde f^p(U) \subset U$ and $\tilde f^{pn} \to \infty$ as $n \to \infty$. Note that periodic Baker domains for entire maps are always simply connected (see \cite{baker75}), while in the transcendental meromorphic case they can be multiply connected. The dynamical properties of Baker domains have been studied in many papers, see
e.g.~\cite{barfag,berdrasin,bergzheng,faghen,konig,rippon,ripponstallardII,ripponstallard} and a survey \cite{rippon-survey}.
   
The Baker--Pommerenke--Cowen results \cite{baker-pommer,cowen,pommer} imply that for a holomorphic map $f: \HH \to \HH$ with $f^n \to \infty$ as $n \to \infty$, there exists a nice simply connected absorbing domain $W$ in $U$ for $f$, such that the map $\varphi$ which semi-conjugates $f$ to a M\"obius map $T: \Omega \to \Omega$ is univalent on $W$. Hence, by the use of a Riemann map, one can construct 
nice simply connected absorbing domains for $f: U \to U$ with $f^n \to \infty$, if $U \subset \C$ is simply connected.  

The existence of such absorbing regions in non-simply connected hyperbolic domains $U$, in particular Baker domains for transcendental meromorphic maps was an open question addressed e.g.~in \cite{bergweiler2,buffruc,mayer}, related to the question of the existence of so-called virtual immediate basins for Newton root-finding algorithm for entire functions. 

In \cite{konig}, K\"onig showed that if $U$ is an arbitrary hyperbolic domain in $\C$ and every closed loop in $U$ is eventually contractible in $U$ under iteration of $f$, then there exists a nice simply connected absorbing domain in $U$ for $f$. In particular, this holds if $U$ is a $p$-periodic Baker domain for a transcendental meromorphic map $\tilde f$ with finitely many poles, where $f = \tilde f^p$ (see Theorem~\ref{thm:konig}). 

In a recent paper \cite{bfjk}, the authors constructed nice absorbing domains for  $f: U \to U$ with $f^n \to \infty$ for an arbitrary hyperbolic domain $U \subset \C$ (see Theorem~\ref{thm:bfjk}). In particular, the construction was used to prove that the Baker domains of Newton's method for entire functions are always simply connected.

In this paper we consider the question of existence of simply connected absorbing domains $W$ in $U$ for $f$. This is equivalent to the condition that every closed loop in $U$ is eventually contractible in $U$ under iteration of $f$ (see Proposition~\ref{prop:simply}). We prove the following.

\begin{thmB} Let $U$ be a hyperbolic domain in $\C$ and let $f: U \to U$ be a holomorphic map, such that $f^n \to \infty$ as $n \to
\infty$ and $\infty$ is not an isolated point of the boundary of $U$ in $\clC$. If $f$ is of parabolic~I type, 
then there exists a nice simply connected absorbing domain $W$ in $U$ for $f$.
\end{thmB}

Note that the assumption on the point at infinity is necessary. In fact, if $\infty$ is an isolated point of the boundary of $U$ in $\clC$, then a simply connected absorbing domain cannot exist for any type of the map $f$ (see Proposition~\ref{prop:isol}).

On the other hand, we provide counterexamples for other types of the map $f$. 

\begin{thmC} There exist transcendental meromorphic maps $f: \C \to \clC$ with an invariant Baker domain $U \subset \C$, such that $f|_U$ is not of parabolic~I type and there is no simply connected absorbing domain $W$ in $U$ for $f$. The examples are constructed in two cases: 
\begin{itemize}
\item $\varrho_U(f^{n+1}(z), f^n(z)) \not\to 0$ as $n \to \infty$ for $z \in U$ and $\inf_{z\in U}\lim_{n\to\infty}\varrho_U(f^{n+1}(z), f^n(z))=0$,
\item $\inf_{z\in U}\lim_{n\to\infty}\varrho_U(f^{n+1}(z), f^n(z))>0$ $(f|_U$ is of hyperbolic type$)$.
\end{itemize}
\end{thmC}

We also provide examples of simply connected absorbing domains $W$ in $U$ for $f$ of parabolic~I type. In all three types of examples, the map $f$ has the form
\[
f(z) = z + 1 + \sum_{p \in \PP} \frac{a_p}{(z-p)^2}, \qquad a_p \in \C \setminus \{0\},
\]
where $\PP$ is the set of poles of $f$. Moreover, $\{\infty\}$ is a singleton component of $\clC\setminus U$, in particular it is a singleton component of $J(f)$. To our knowledge, these are the first examples of Baker domains of that kind. A detailed description of the examples is contained in Theorem~\ref{thm:ex}.

The plan of the paper is the following. In Section~\ref{sec:back} we present definitions and results used in the proofs of Theorems~A,~B,~C. In Section~\ref{sec:parab} we characterize parabolic~I type (Theorem~A) and in Section~\ref{sec:absorb} we prove Theorem~B. The examples described in Theorem~C are constructed in Section~\ref{sec:ex}.

\section*{Acknowledgements}
We wish to thank the Institut de Matem\`atica de la Universitat de Barcelona (IMUB)  for their hospitality.

\section{Background}\label{sec:back}

For $z \in \C$ and $A, B \subset \C$ we write 
\[
\dist(z, A) = \inf_{a \in A}|z - a|, \quad  \dist(A, B) = \inf_{a \in A,\: b \in B}|a - b|.
\]
The symbols $\overline{A}$, $\bd A$ denote, respectively, the closure and boundary of $A$ in $\C$. The Euclidean disc of radius $r$ centred at $z \in\C$ is denoted by $\D(z,r)$ and the unit disc $\D(0,1)$ is simply written as $\D$.

Let $U \subset \C$ be a hyperbolic domain, i.e.~a domain whose complement in $\C$ contains at least two points. By the Uniformization Theorem, there exists a
universal holomorphic covering $\pi$ from the right half-plane $\HH$ onto $U$. Every holomorphic map $f: U \to U$ can be lifted by $\pi$ to a holomorphic map $g: \HH \to \HH$, such that the diagram
\[
\begin{CD}
\HH @>g>> \HH\\
@VV\pi V @VV\pi V\\
U @>f>> U
\end{CD}
\]
commutes. By $\varrho_U(\cdot)$ and $\varrho_U(\cdot, \cdot)$ we denote, respectively, the density of the hyperbolic metric and the hyperbolic distance in $U$, defined by the use of the hyperbolic metric in $\HH$. The disc of radius $r$ centred at $z$ with respect to the hyperbolic metric in $U$ is denoted by $\DD_U(z,r)$.

Recall the classical Schwarz--Pick Lemma and Denjoy--Wolff Theorem.

\begin{Schwarzlem}[{\cite[Theorem 4.1]{carlesongamelin}}]
Let $U, V$ be hyperbolic domains in $\C$ and
let $f: U \to V$ be a holomorphic map. Then
\[
\varrho_V(f(z), f(z')) \leq \varrho_U(z, z')
\]
for every $z, z' \in U$. In particular, if $U \subset V$, then
\[
\varrho_V(z, z') \leq \varrho_U(z, z'),
\]
with strict inequality unless $z = z'$ or $f$ lifts to a M\"obius automorphism of 
$\HH$.
\end{Schwarzlem}

\begin{DWthm}[{\cite[Theorem 3.1]{carlesongamelin}}] 
Let $g: \HH \to \HH$ be a non-constant holomorphic map, which is not a M\"obius automorphism of $\HH$. Then there exists a point $\zeta \in \overline{\HH} \cup \{\infty\}$, such that $g^n$ tends to $\zeta$ as $n \to \infty$ almost uniformly on $\HH$.
\end{DWthm}

The following estimate relates the hyperbolic density $\varrho_U$ to the quasi-hyperbolic density $1/\dist(z, \bd U)$.

\begin{lem}[{\cite[Theorem 4.3]{carlesongamelin}}] \label{lem:CG} Let $U \subset \C$ be a hyperbolic domain. Then
\[
\varrho_U(z) \leq \frac{2}{\dist(z, \bd U)} \qquad \text{for } z \in U
\]
and
\[
\varrho_U(z) \geq \frac{1 + o(1)}{\dist(z, \bd U) \log (1/\dist(z, \bd
U))}\qquad \text{as } z \to \bd U.
\]
Moreover, if $U$ is simply connected, then
\[
\varrho_U(z) \geq \frac{1}{2\dist(z, \bd U)}\qquad \text{for } z \in U.
\]
\end{lem}

The above lemma implies the following standard estimation of the hyperbolic distance. We include the proof for completeness.

\begin{lem}\label{lem:hyp} Let $U$ be a hyperbolic domain in $\C$ and let $z, z' \in U$. Then
\[
\frac{|z-z'|}{\dist(z, \bd U)}\ge 1 - e^{-\varrho_U(z, z')/2}.
\]
\end{lem}
\begin{proof} Suppose that there exist $z,z'\in U$ such that  
\begin{equation}\label{eq:<}
\frac{|z-z'|}{\dist(z, \bd U)} < 1 - e^{-\varrho_U(z, z')/2}
\end{equation}
and let $\gamma$ be the straight line segment connecting $z$ and
$z'$. In particular, \eqref{eq:<} implies that $|z-z'| < \dist(z, \bd U)$, so $\gamma \subset U$ and $|u-z| < \dist(z, \bd U)$ for $u \in \gamma$. Thus, by the second inequality in Lemma~\ref{lem:CG},
\begin{multline*}
\varrho_U(z, z') \leq \int_\gamma \varrho_U(u) |du| \leq 2
\int_\gamma \frac{|du|}{\dist(u,\bd U)}
\leq 2\int_\gamma \frac{|du|}{\dist(z,\bd U) - |u - z|}\\
= 2\int_0^{|z-z'|} \frac{ds}{\dist(z,\bd U) - s} = 
2 \ln \frac{\dist(z,\bd U)}{\dist(z,\bd U) - |z-z'|}
= -2\ln \left(1 - \frac{|z-z'|}{\dist(z,\bd U)}\right),
\end{multline*}
which contradicts \eqref{eq:<}.
\end{proof}

The lower bounds for hyperbolic metric from Lemma~\ref{lem:CG} can be improved in the presence of dynamics. The following result was proved by Rippon in \cite{rippon} (actually, it was formulated under an additional assumption $f^n \to \infty$ as $n\to\infty$, but the proof does not use this).

\begin{thm}[{\cite[Theorem 1]{rippon}}]\label{thm:rippon}
Let $U$ be a hyperbolic domain in $\C$ and let $f: U \to U$ be a holomorphic map without fixed points and without isolated boundary fixed points. Then for every compact set $K \subset U$ there exists a constant $C > 0$ such that 
\[
\frac{|f^n(z) - f^n(z')|}{\dist(f^n(z), \bd U)} \leq C \varrho_U(f^n(z), f^n(z'))
\]
for every $z, z' \in K$ and every $n \geq 0$.
\end{thm}

The following result proved by Bonfert in \cite{bonfert} describes
a relation between the dynamical behaviour of $f$ and its lift $g$.

\begin{thm}[{\cite[Theorem 1.1]{bonfert}}]
\label{thm:bonfert1.1}
Let $U$ be a hyperbolic domain in $\C$ and let $f: U \to U$ be a holomorphic
map without fixed points. Let $g: \HH \to \HH$ be a lift of $f$ by a universal
covering map $\pi: \HH \to U$. Then
\[
\varrho_U(f^{n+1}(z), f^n(z)) \to 0 \iff \varrho_\HH(g^{n+1}(1), g^n(1))\to 0
\]
as $n\to\infty$ for any $z \in U$.
\end{thm}

An easy consequence of this theorem is that the left hand side of the equation is either satisfied for every $z\in U$ or for none.

The next theorem summarizes the results of Baker--Pommerenke--Cowen \cite{baker-pommer,cowen,pommer} on the dynamics of holomorphic maps in $\HH$. We use the notation from \cite{cowen}. (The equivalence of the Baker--Pommerenke and Cowen approaches were shown by K\"onig in \cite{konig}.) 

\begin{Cowenthm}[{\cite[Theorem 3.2]{cowen}, see also \cite[Lemma 1]{konig}}] 
Let $g: \HH \to \HH$ be a holomorphic map such that $g^n \to\infty$ as $n \to \infty$. Then there exists a simply connected domain $V \subset \HH$, a domain $\Omega$ equal to $\HH$ or $\C$, a holomorphic map $\varphi: \HH \to \Omega$, and a M\"obius transformation $T$ mapping $\Omega$ onto itself, such that:
\begin{itemize}
\item[$(a)$] $V$ is absorbing in $\HH$ for $g$,
\item[$(b)$] $\varphi(V)$ is absorbing in $\Omega$ for $T$,
\item[$(c)$] $\varphi \circ g = T \circ \varphi$ on $\HH$,
\item[$(d)$] $\varphi$ is univalent on $V$.
\end{itemize}
Moreover, $\varphi$, $T$ depend only on $g$. In fact $($up to
a conjugation of $T$ by a M\"obius transformation preserving $\Omega)$, one of
the following cases holds:
\begin{itemize}
\item $\Omega = \C$, $T(\omega) = \omega + 1$ $($parabolic~I type $)$,
\item $\Omega = \HH$, $T(\omega) = \omega \pm i$ $($parabolic~II type$)$,
\item $\Omega = \HH$, $T(\omega) = a\omega$ for some $a > 1$ $($hyperbolic type$)$.
\end{itemize}
\end{Cowenthm}

\begin{rem*}
An equivalent description of the three cases can be done by taking 
\begin{itemize}
\item $\Omega = \C$ $($parabolic~I type $)$,
\item $\Omega = \{z\in \C: \Im (z) > 0\}$ $($parabolic~II type$)$,
\item $\Omega = \{z\in \C: 0 < \Im (z) < b\}$ for some $b > 0$ $($hyperbolic type$)$
\end{itemize}
and $T(\omega) = \omega + 1$ in all three cases.
\end{rem*}

The following theorem gathers the K\"onig results from \cite{konig}.

\begin{thm}[\cite{konig}]\label{thm:konig}
Let $U$ be a hyperbolic domain in $\C$ and let $f: U \to U$ be a holomorphic
map, such that $f^n \to \infty$ as $n \to \infty$. Suppose that for every
closed curve $\gamma \subset U$ there exists $n > 0$ such that $f^n(\gamma)$ is
contractible in $U$. Then there exists a simply
connected domain $W \subset U$, a domain $\Omega$ and a transformation $T$ as
in Cowen's Theorem, and a holomorphic map $\psi: U \to \Omega$, such that:
\begin{itemize}
\item[$(a)$] $W$ is absorbing in $U$ for $f$,
\item[$(b)$]  $\psi(W)$ is absorbing in $\Omega$ for $T$,
\item[$(c)$] $\psi \circ f = T \circ \psi$ on $U$,
\item[$(d)$] $\psi$ is univalent on $W$.
\end{itemize}
Moreover, 
\begin{itemize}
\item 
$T$ is of parabolic~I type if and only if 
\[
\lim_{n\to\infty} \frac{|f^{n+1}(z)-f^n(z)|}{\dist(f^n(z), \bd U)} = 0 \quad \text{for every } z\in U,
\]
\item 
$T$ is of parabolic~II type if and only if 
\[
\varliminf_{n\to\infty} \frac{|f^{n+1}(z)-f^n(z)|}{\dist(f^n(z), \bd U)} > 0  \text{ for every } z\in U 
\quad \text{and} \quad
\inf_{z\in U}\varlimsup_{n\to\infty} \frac{|f^{n+1}(z)-f^n(z)|}{\dist(f^n(z), \bd U)} =0, 
\]
\item
$T$ is of hyperbolic type if and only if 
\[
\inf_{z\in U} \inf_{n\geq 0} \frac{|f^{n+1}(z)-f^n(z)|}{\dist(f^n(z), \bd U)} > 0.
\]
\end{itemize}

Furthermore, if $\tilde f:\C \to \clC$ is a meromorphic map with finitely many poles, and $U$ is a periodic Baker domain of period $p$, then the above assumptions are satisfied for $f = \tilde f^p$, and consequently, there exists a simply connected domain $W$ in $U$ with the properties $(a)$--$(d)$ for $f = \tilde f^p$.
\end{thm}

In fact, if under the assumptions of Theorem~\ref{thm:konig}, we take $V$ and $\varphi$ from Cowen's Theorem for a lift $g$ of $f$ by a universal covering $\pi : \HH \to U$, then $\pi$ is univalent in $V$ and one can take $W = \pi(V)$ and $\psi = \varphi \circ \pi^{-1}$, which is well defined in $U$.

\begin{rem} \label{rem:nice} It was shown by the authors in \cite{bfjk} that under the conditions of Cowen's Theorem or Theorem~\ref{thm:konig}, one can choose the absorbing domain $W$ to be nice. 
\end{rem}

\begin{defn}[\bf Isolated boundary fixed point] \label{ibfp} Let $\zeta$ be an isolated point $\zeta$ of the boundary of $U$ in $\clC$, i.e. there exists a neighborhood $V\subset \clC$ of $\zeta$ such that $V\setminus \{\zeta\} \subset U$. Since $f(U)\subset U$ and $U$ is hyperbolic, it follows from Picard's Theorem that $f$ extends holomorphically to $V\cup \{\zeta\}$. If $f(\zeta)=\zeta$, we say that $\zeta$ is an {\em isolated boundary fixed point of $f$}. 
\end{defn}
 
The following theorem was proved by Bonfert in \cite{bonfert}.

\begin{thm}[{\cite[Theorem 1.4]{bonfert}}]\label{thm:bonfert1.4}
Let $U$ be a hyperbolic domain in $\C$ and let $f: U \to U$ be a holomorphic
map without fixed points and without isolated boundary fixed points. Then
there exist a domain $X\subset \C$, a non-constant holomorphic map $\Psi: U \to
X$ and a M\"obius transformation $S$ mapping $X$ onto itself, such that 
\[
\Psi\circ f = S \circ \Psi.
\]
Moreover, if $\varrho_U(f^{n+1}(z), f^n(z)) \to 0$ for $z\in U$, then
$X = \C$, $S(\omega) = \omega + 1$. Otherwise, the domain $X$ is
hyperbolic. 
\end{thm}

The existence of nice absorbing regions in arbitrary hyperbolic domains was proved by the authors in \cite{bfjk}. 

\begin{thm}[{\cite[Theorem A]{bfjk}}]\label{thm:bfjk}
Let $U$ be a hyperbolic domain in $\C$ and let $f: U \to U$ be a holomorphic map, such that $f^n \to \infty$ as $n \to
\infty$. Then there exists a nice absorbing domain $W$ in $U$ for $f$, such that $f$ is locally univalent on $W$. Moreover, for every $z \in U$ and every sequence of positive numbers $r_n$, $n \geq 0$ with $\lim_{n\to\infty} r_n = \infty$,
the domain $W$ can be chosen such that
\[
W \subset \bigcup_{n=0}^\infty \DD_U(f^n(z), r_n).
\]
\end{thm}

\begin{rem*} If $f$ is of parabolic~{\rm I} type, then $W$ can be chosen such that  
$W \subset \bigcup_{n=0}^\infty \DD_U(f^n(z), b_n)$ 
for a sequence $b_n$ with $\lim_{n\to\infty} b_n = 0$ and $b_n < b$ for an arbitrary given $b > 0$, see \cite[Proposition 3.1]{bfjk}.
\end{rem*}

\section{Characterization of parabolic~I type: Proof of Theorem A}\label{sec:parab}

Let $U$ be a hyperbolic domain in $\C$ and let $f: U \to U$ be a holomorphic map without fixed points. Consider a universal covering $\pi : \HH \to U$ and a lift $g: \HH \to \HH$ of the map $f$ by $\pi$. 
Then $g$ has no fixed points, so by the Denjoy--Wolff Theorem, $g^n \to \zeta$ for a point $\zeta$ in the boundary of $\HH$ in $\clC$. Conjugating $g$ by a M\"obius map, we can assume $\zeta = \infty$. Consider the map $T: \Omega \to \Omega$ from Cowen's Theorem for the map $g$. By properties of a universal covering, for different choices of $\pi$ and $g$, the suitable maps $T$ are conformally conjugate, so in fact the type of $T$ does not depend on the choice of $\pi$ and $g$. Hence, we can state the following definition.

\begin{defn*}[{\bf\boldmath Type of $f$}{}] The map $f$ is of parabolic~I, parabolic~II or hyperbolic type if the same holds for its lift $g$.
\end{defn*}

\begin{thmA}
Let $U$ be a hyperbolic domain in $\C$ and let $f: U \to U$ be a holomorphic
map without fixed points and without isolated boundary fixed points. Then the
following statements are equivalent:
\begin{itemize}
\item[$(a)$] $f$ is of parabolic~I type,
\item[$(b)$] $\varrho_U(f^{n+1}(z), f^n(z)) \to
0$ as $n \to \infty$ for some $z \in U$, 
\item[$(c)$] $\varrho_U(f^{n+1}(z), f^n(z)) \to
0$ as $n \to \infty$ almost uniformly on $U$, 
\item[$(d)$] $|f^{n+1}(z)- f^n(z)|/\dist(f^n(z), \bd U) \to
0$ as $n \to \infty$ for some $z \in U$,
\item[$(e)$] $|f^{n+1}(z)- f^n(z)|/\dist(f^n(z), \bd U) \to
0$ as $n \to \infty$ almost uniformly on $U$.
\end{itemize}
\end{thmA}

\begin{proof} 
First we prove \mbox{(a) $\Rightarrow$ (b)}. Since $f$ is of parabolic~I type, we have $\Omega = \C$, $T(\omega) = \omega +1$ in
Cowen's Theorem for a lifted map $g: \HH \to \HH$. We claim that for every $\omega \in \C$ there exists $m \in\N$
and a sequence $d_n > 0$ with $d_n\to\infty$ as $n \to \infty$, such that 
\begin{equation}
\label{eq:D}
\D(T^n(\omega), d_n) \subset \varphi(V) \qquad \text{for every } n \ge m,
\end{equation}
for $\varphi$, $V$ from Cowen's Theorem for $g$. 
Indeed, if \eqref{eq:D} does not hold, then $\D(T^n(\omega), d) \not\subset \varphi(V)$ for some $d > 0$ and infinitely many $n$, which contradicts assertion~(b) of Cowen's Theorem for $K = \overline{\D(\omega, d)}$. 

Since $|T^{n+1}(\omega) - T^n(\omega)| = |\omega + n+ 1 - (\omega + n)| = 1$ and $d_n \to \infty$, Lemma~\ref{lem:CG} implies 
\[
\varrho_{\D(T^n(\omega), d_n)}(T^{n+1}(\omega),
T^n(\omega)) \to 0,
\]
so by \eqref{eq:D} and the Schwarz--Pick Lemma,
\[
\varrho_{\varphi(V)}(T^{n+1}(\omega),
T^n(\omega)) \leq \varrho_{\D(T^n(\omega), d_n)}(T^{n+1}(\omega),
T^n(\omega)) \to 0
\]
as $n \to \infty$. By Cowen's Theorem, $\varphi$ is univalent on $V$, so for $\omega \in \varphi(V)$ and $z =  \pi((\varphi|_V)^{-1}(\omega))$, by the Schwarz--Pick Lemma applied to  the map $\pi\circ(\varphi|_V)^{-1}$, we have
\[
\varrho_U(f^{n+1}(z), f^n(z)) \leq \varrho_{\varphi(V)}(T^{n+1}(\omega),
T^n(\omega))\to 0,
\]
which shows (b). 

Now we show \mbox{(b) $\Rightarrow$ (a)}. By Theorem~\ref{thm:bonfert1.1}, we
have $\varrho_\HH(g^{n+1}(w), g^n(w)) \to 0$ for some $w \in \HH$. Suppose that $g$ is not of parabolic~I type. Then $\Omega = \HH$ in Cowen's Theorem for the map $g$, so by the Schwarz--Pick Lemma applied to the map $\varphi$, we have 
\[
\varrho_\HH(T^{n+1}(\varphi(w)), T^n(\varphi(w))) = \varrho_\HH(\varphi(g^{n+1}(w)),
\varphi(g^n(w))) \leq \varrho_\HH(g^{n+1}(w),
g^n(w)) \to 0,
\]
which is not possible, since for $T(\omega) = a \omega$ or $T(\omega) = \omega
\pm i$ we have 
\[
\varrho_\HH(T^{n+1}(\varphi(w)), T^n(\varphi(w))) =
\varrho_\HH(T(\varphi(w)), \varphi(w)) >0.
\]
Hence, $\Omega = \C$, $T(\omega) = \omega +1$ and $g$ is of parabolic~I type. 

The implication \mbox{(c) $\Rightarrow$ (b)} is trivial. To show \mbox{(b) $\Rightarrow$ (c)}, note first that by Theorem~\ref{thm:bonfert1.1}, the pointwise convergence holds for every $z \in U$. Take a compact set $K \subset U$ and suppose that the convergence is not uniform on $K$. This means that there exist sequences $z_j \in K$, $n_j \to \infty$ as $j \to \infty$, and a constant $c > 0$, such that 
\[
\varrho_U(f^{n_j+1}(z_j), f^{n_j}(z_j)) > c.
\]
Passing to a subsequence, we can assume $z_j \to z$  for some $z \in K$. Then
$\varrho_U(z_j, z) \to 0$, so by the Schwarz--Pick Lemma, for every $n \geq 0$
\[
\varrho_U(f^n(z_j), f^n(z)) \leq \varrho_U(z_j, z)\to 0
\]
as $j \to \infty$. Hence, since $\varrho_U(f^{n_j+1}(z), f^{n_j}(z)) \to 0$ by the pointwise convergence, we have 
\begin{align*}
0 < c &< \varrho_U(f^{n_j+1}(z_j), f^{n_j}(z_j)) \\
&\leq \varrho_U(f^{n_j+1}(z_j), f^{n_j+1}(z)) + \varrho_U(f^{n_j+1}(z), f^{n_j}(z)) + \varrho_U(f^{n_j}(z), f^{n_j}(z_j))\\
&\leq 2 \varrho_U(z_j, z) + \varrho_U(f^{n_j+1}(z), f^{n_j}(z)) \to 0,
\end{align*}
which is a contradiction. This ends the proof of \mbox{(b) $\Leftrightarrow$ (c)}.

The implication \mbox{(c) $\Rightarrow$ (e)} follows from Theorem~\ref{thm:rippon} for $z' = f(z)$ and the implication \mbox{(e) $\Rightarrow$ (d)} is trivial. 
To show \mbox{(d) $\Rightarrow$ (b)}, we use Lemma~\ref{lem:hyp} for the points $f^n(z)$, $f^{n+1}(z)$. 
In this way we have proved the equivalences (b) $\Leftrightarrow$ (c) $\Leftrightarrow$ (d) $\Leftrightarrow$ (e). 
\end{proof}

By Theorem~A and Theorem~\ref{thm:bonfert1.4}, we immediately obtain the following.

\begin{corA'}
The map $f$ is of parabolic~I type if and only if we have $X = \C$, $S(\omega) = \omega +1$ in Theorem~{\rm \ref{thm:bonfert1.4}}.
\end{corA'}

The following proposition gives a sufficient condition for a map $f$ to be of hyperbolic type.

\begin{prop}\label{prop:hyper}
Let $U$ be a hyperbolic domain in $\C$ and let $f: U \to U$ be a holomorphic
map without fixed points and without isolated boundary fixed points. If
\[
\inf_{z\in U}\lim_{n\to\infty}\varrho_U(f^{n+1}(z), f^n(z)) > 0,
\]
then $f$ is of hyperbolic type.
\end{prop}

\begin{rem}\label{rem:limit}
By the Schwarz--Pick Lemma, the limit $\lim_{n\to\infty}\varrho_U(f^{n+1}(z), f^n(z))$ always exists for $z \in U$ and 
\[
\inf_{z\in U}\lim_{n\to\infty}\varrho_U(f^{n+1}(z), f^n(z)) > 0 \iff 
\inf_{z\in U} \varrho_U(f(z), z) > 0.
\]
\end{rem}

\begin{proof}[Proof of Proposition~\rm\ref{prop:hyper}] In view of Theorem~A, to prove the proposition it is sufficient to show that if $f$ is of parabolic~II type, then 
\[
\inf_{z\in U}\lim_{n\to\infty} \varrho_U(f^{n+1}(z), f^n(z)) =0. 
\]
We proceed as in the proof of the implication (a) $\Rightarrow$ (b) in Theorem~A.
Let $g: \HH \to \HH$ be a lift of $f$ by a universal covering $\pi: \HH \to U$. If $f$ is of parabolic~II type, then we have $\Omega = \HH$, $T(\omega) = \omega \pm i$ in Cowen's Theorem for the map $g$.

Take a small $\varepsilon > 0$ and $\omega \in \HH$ with $\Re(\omega) = 1/\varepsilon$. Let
\[
D_n = \D\left(T^n(\omega), \frac{1}{2\varepsilon}\right)
\]
for $n \ge 0$. Then $\overline{D_n} \subset \HH$, so $D_n \subset \varphi(V)$ for large $n =n(\varepsilon,\omega)$ by the assertion~(b) of Cowen's Theorem. We have $|T^{n+1}(\omega) - T^n(\omega)| = 1$ and 
$\dist(T^n(\omega), \bd D_n) = 1/(2\varepsilon)$, so Lemma~\ref{lem:CG} implies 
\[
\varrho_{D_n}(T^{n+1}(\omega), T^n(\omega)) < 4\varepsilon.
\]
Since $\varphi$ is univalent on $V$, by the Schwarz--Pick Lemma we have
\[
\varrho_U(f^{n+1}(z), f^n(z)) \leq 
\varrho_{\varphi(V)}(T^{n+1}(\omega), T^n(\omega)) \leq
\varrho_{D_n}(T^{n+1}(\omega), T^n(\omega)) < 4\varepsilon
\]
for large $n$, where $z = \pi((\varphi|_V)^{-1}(\omega))$. Hence, for any  arbitrarily small $\varepsilon > 0$, there exists $z\in U$ such that $\varrho_U(f^{n+1}(z), f^n(z))< 4\varepsilon$ for $n$ large enough. It follows that 
\[
\inf_{z\in U}\lim_{n\to\infty} \varrho_U(f^{n+1}(z), f^n(z)) =0. 
\]
\end{proof}

\begin{rem}\label{rem:konig} If the images under $f^n$ of any closed curve in $U$ are eventually contractible in $U$ (e.g.~when $U$ is a Baker domain for a meromorphic map with finitely many poles), then by Theorem~\ref{thm:konig}, one obtains a characterization of all three types of $f$ in terms of its dynamical behaviour. In the general case, apart of the characterization of parabolic~I type in Theorem~A, Proposition~\ref{prop:hyper} gives a sufficient condition for $f$ to be of hyperbolic type. It would be interesting to check whether the condition is necessary, and to obtain a characterization of all three types of $f$ in terms of its dynamical behaviour in the general case.
\end{rem}

\section{Simply connected absorbing domains: Proof of Theorem B}\label{sec:absorb}

With the goal of proving Theorem B,  we present a condition equivalent to the existence of a simply connected absorbing domain $W$ in $U$ for $f$.

\begin{prop} \label{prop:simply} Let $U$ be a hyperbolic domain in $\C$ and let $f: U \to U$ be a holomorphic map, such that $f^n \to \infty$ as $n \to
\infty$. Then the following statements are equivalent:
\begin{itemize}
\item[$(a)$] There exists a simply connected absorbing domain $W$ in $U$ for $f$.
\item[$(b)$] There exists a nice simply connected absorbing domain $W$ in $U$ for $f$.
\item[$(c)$] For every closed curve $\gamma \subset U$ there exists $n > 0$ such that $f^n(\gamma)$ is contractible in $U$.
\end{itemize}
\end{prop}
\begin{proof}
The implication \mbox{(a) $\Rightarrow$ (c)} follows by the absorbing property of $W$, the implication \mbox{(c) $\Rightarrow$ (b)} is due to Theorem~\ref{thm:konig} and Remark~\ref{rem:nice}, and the implication \mbox{(b) $\Rightarrow$ (a)} is trivial.
\end{proof}

\begin{defn*}
For a compact set $X \subset \C$ we denote by $\ext(X)$
the connected component of $\clC \setminus X$ containing infinity.
We set $K(X) = \clC \setminus \ext(X)$.
\end{defn*}

Before proving Theorem~B, we show that if $\infty$ is an isolated point of the boundary of $U$, then a simply connected absorbing domain $W$ does not exist.

\begin{prop}\label{prop:isol} Let $U$ be a hyperbolic domain in $\C$ and let $f: U \to U$ be a holomorphic map, such that $f^n \to \infty$ as $n \to
\infty$. If $\infty$ is an isolated point of the boundary of $U$ in $\clC$, then there is no simply connected absorbing domain $W$ in $U$ for $f$.
\end{prop}
\begin{proof}
By assumption, $U$ is a punctured neighbourhood of $\infty$ in $\clC$. Since $U$ is hyperbolic and $f(U) \subset U$, the set $f(U)$ omits at least three points in $\clC$, so by the Picard Theorem, the map $f$ extends holomorphically to $U \cup \{\infty\}$. Let $V = \{z\in\C: |z| > R\} \cup \{\infty\}$ for a large $R > 0$. Since $f^n \to \infty$ uniformly on the boundary of $V$ in $\clC$, by the openness of $f^n$, the closure of $f^n(V)$ in $\clC$ is contained in $V$ for every sufficiently large $n$. This easily implies that $\infty$ is an attracting fixed point for the extended map $f$. Hence, $f$ in a neighbourhood of $\infty$ is conformally conjugate by a map $\psi$ to $z \mapsto \lambda z$ for some $\lambda \in \C$, $0 < |\lambda| < 1$ or to $z \mapsto z^k$ for some integer $k \geq 2$ in a neighbourhood of $0$. Let $\gamma = \psi^{-1}(\bd\D(0, r))$ for a small $r > 0$. Then for every $n > 0$ we have  $f^n(\gamma) \subset \ext(\gamma)$ and $K(f^n(\gamma)) \supset K(\gamma)$, so $f^n(\gamma)$ is 
not contractible in $U$ and we can use Proposition~\ref{prop:simply} to end the proof.
\end{proof}

Now we prove the main result of this section.

\begin{thmB}Let $U$ be a hyperbolic domain in $\C$, let $f: U \to U$ be a holomorphic map, such that $f^n(z) \to \infty$ as $n \to
\infty$ for $z \in U$ and $\infty$ is not an isolated point of the boundary of $U$ in $\clC$. If $f$ is of parabolic~I type, 
then there exists a nice simply connected absorbing domain $W$ in $U$ for $f$.
\end{thmB}

\begin{proof}
Note first that $f$ has no fixed points in $U$. Moreover, we will show that $f$ has no isolated boundary fixed points. Indeed, suppose that $\zeta_0$ is an isolated point of the boundary of $U$ in $\clC$ and $f$ extends holomorphically to $U \cup \{\zeta_0\}$ with $f(\zeta_0) = \zeta_0$. By assumption, $\zeta_0 \neq \infty$. Take a Jordan curve $\gamma_0 \subset U$ surrounding $\zeta_0$ in a small neighbourhood of $\zeta_0$, such that $D \setminus \{\zeta_0\} \subset U$, where $D$ is the component of $\C\setminus \gamma_0$ containing $\zeta_0$. Since $f^n \to \infty$ uniformly on $\gamma_0$ and $\zeta_0 = f^n(\zeta_0) \in f^n(D)$, by the Maximum Principle we obtain 
\[
U \supset \bigcup_{n = 0}^\infty f^n(D \setminus \{\zeta_0\}) \supset \bigcup_{n = 0}^\infty f^n(D) \setminus \{\zeta_0\} = \C \setminus \{\zeta_0\},
\]
so in fact $U = \C \setminus \{\zeta_0\}$, which is impossible since $U$ is hyperbolic.
Hence, $f$ has no isolated boundary fixed points.

Take a closed curve $\gamma \subset U$. We will check that there exists $n > 0$ such that $f^n(\gamma)$ is contractible in $U$.  By Proposition~\ref{prop:simply}, this will prove the existence of a nice simply connected absorbing domain $W$ in $U$ for $f$. Suppose this is not true. By Theorem~\ref{thm:rippon} for $K = \gamma \cup f(\gamma)$, there exists a constant $C > 0$ such that for every $z \in \gamma$ and every $n \geq 0$,
\[
\frac{|f^{n+1}(z) - f^n(z)|}{\dist(f^n(z), \bd U)} \leq C \varrho_U(f^{n+1}(z), f^n(z)),
\]
so by Theorem~A~(c), there exists $n_0 \geq 0$ such that for every $z \in \gamma$ and every $n \geq n_0$,
\begin{equation}\label{eq:dist}
|f^{n+1}(z) - f^n(z)| < \frac{1}{2}\dist(f^n(z), \bd U).
\end{equation}

Fix an arbitrary point $v \in \C \setminus U$. By \eqref{eq:dist}, for every $z \in \gamma$ and every $n \geq n_0$ we have 
\[
|f^{n+1}(z) - f^n(z)| < \frac{1}{2} |f^n(z) - v|.
\]
This implies that $f^{n+1}(z)-v$ lies in a disc D of center $f^n(z)-v$ and radius $\frac12 |f^n(z)-v|$. Clearly $0\notin D$ and a simple calculation shows that $D$ is included in a sector of angle $\pi/6$. Hence, there exists a branch of the argument function in $D$, and in a small neighborhood of $z$ we have 
\[
|\Arg(f^{n+1}(z) - v) - \Arg(f^n(z) - v)| < \frac{\pi}{6}.
\]
Taking an analytic continuation of this branch while $z$ goes around $\gamma$, we see that the winding number of $f^n(\gamma)$ around $v$ is the same as the winding number of $f^{n+1}(\gamma)$ around $v$. In particular, $v$ is in a bounded component of $\C \setminus f^n(\gamma)$ if and only if it is in a bounded component of $\C \setminus f^{n+1}(\gamma)$. By induction, this implies
\begin{equation}
\label{eq:v}
v \in K(f^n(\gamma)) \quad \text{if and only if} \quad v \in K(f^{n'}(\gamma))
\end{equation}
for every $n' \geq n$.

Take $n \geq n_0$. By assumption, $f^n(\gamma)$ is not contractible in $U$, so there exists a point $v_0 \in K(f^n(\gamma)) \setminus U$. Since $f^k \to \infty$ as $k \to \infty$ uniformly on $\gamma$, there exists $n'> n$ such that $f^{n'}(\gamma) \cap K(f^n(\gamma)) = \emptyset$. We cannot have $K(f^{n'}(\gamma)) \subset \ext(f^n(\gamma))$, because then we would have $v_0 \notin K(f^{n'}(\gamma))$, which contradicts \eqref{eq:v} for $v = v_0$. Hence, $K(f^n(\gamma)) \subset K(f^{n'}(\gamma))$, so
\[
K(f^n(\gamma)) \setminus U \subset K(f^{n'}(\gamma)) \setminus U.
\]
On the other hand, 
\[
K(f^n(\gamma))\setminus U \supset K(f^{n'}(\gamma)) \setminus U,
\]
because otherwise there would exists a point $v_1 \in \C \setminus U$, such that 
$v_1 \in K(f^{n'}(\gamma)) \setminus K(f^n(\gamma))$, which contradicts \eqref{eq:v} for $v = v_1$. 

We conclude that for every $n \geq n_0$ there exists $n' > n$, such that
\[
K(f^n(\gamma)) \setminus U = K(f^{n'}(\gamma)) \setminus U.
\]
In this way we can construct an increasing sequence $n_j$, $j \geq 0$, such that 
\begin{equation}\label{eq:K}
K(f^{n_j}(\gamma))\subset K(f^{n_{j+1}}(\gamma)) \qquad \text{and} \qquad
K(f^{n_j}(\gamma))\setminus U = K(f^{n_0}(\gamma)) \setminus U
\end{equation}
for every $j$. Since $f^{n_j} \to \infty$ as $j \to \infty$ uniformly on $\gamma$, \eqref{eq:K} implies that  
\[
\bigcup_{j = 0}^\infty K(f^{n_j}(\gamma)) = \C
\]
and $\C \setminus U =\bigcup_{j=0}^\infty K(f^{n_j}(\gamma))\setminus U =  K(f^{n_0}(\gamma)) \setminus U$ is a compact subset of $\C$. Hence, $U$ contains a punctured neighbourhood of $\infty$ in $\clC$, so $\infty$ is an isolated point of the boundary of $U$ in $\clC$, which is a contradiction. 
\end{proof}

\section{Examples}\label{sec:ex}

Let
\begin{equation}\label{eq:f}
f(z) = z + 1 + \sum_{p \in \PP} \frac{a_p}{(z-p)^2}, \qquad a_p \in \C \setminus \{0\},
\end{equation}
where $\PP \subset \C$ has one of the three following forms:
\begin{itemize}
\item[(i)] $\PP = i\Z = \{im: m \in \Z\}$,
\item[(ii)] $\PP = \Z$ or $\Z_+$,
\item[(iii)] $\PP = \Z + i\Z= \{j+im: j,m \in \Z\}$.
\end{itemize}
It is obvious that for sufficiently small $|a_p|$, the map \eqref{eq:f} is transcendental meromorphic, with the set of poles equal to $\PP$.

Let 
\[
\tilde \PP = \bigcup_{j=0}^\infty (\PP - j) = \{p-j: p\in \PP, j =0, 1, \ldots\}.
\]

(In the case $\PP = \Z$ or $\PP = \Z + i\Z$ we have $\tilde \PP = \PP$.) The assertions of Theorem~C and other results mentioned in Section~\ref{sec:intro} are gathered in the following theorem.

\begin{thm}\label{thm:ex} For every sufficiently small $\delta > 0$, there exists a map $f$ of the form \eqref{eq:f} with an invariant Baker domain $U$, such that $U \supset \C \setminus \bigcup_{p\in\tilde\PP} \D(p, \delta)$. Moreover,
\begin{itemize}
\item in case~$(i)$ we have $\varrho_U(f^{n+1}(z), f^n(z)) \to 0$ as $n \to \infty$ for $z \in U$ and $f|_U$ is of parabolic~I type, so there exists a simply connected absorbing domain $W$ in $U$ for $f$;
\item
in case~$(ii)$ we have $\varrho_U(f^{n+1}(z), f^n(z)) \not\to 0$ as $n \to \infty$ for $z \in U$, 
\[
\inf_{z\in U} \lim_{n\to\infty} \varrho_U(f^{n+1}(z), f^n(z))=0,
\]
and there does not exist a simply connected absorbing domain $W$ in $U$ for $f$;
\item in case~$(iii)$ we have 
\[
\inf_{z\in U}\lim_{n\to\infty}\varrho_U(f^{n+1}(z), f^n(z))>0,
\]
$f|_U$ is of hyperbolic type, and there does not exist a simply connected absorbing domain $W$ in $U$ for $f$.
\end{itemize}
Moreover, in all three cases, $\{\infty\}$ is a singleton component of $\clC\setminus U$, in particular it is a singleton component of $J(f)$.
\end{thm}

Let
\[
e(z) = \sum_{p \in \PP} \frac{a_p}{(z-p)^2} = f(z)-(z+1) 
\]
and notice for further purposes that for all $n\geq 1$,
\[ 
f^n(z) - (z+n)=\sum_{k=0}^{n-1} e(f^k(z)).
\]

To prove Theorem~\ref{thm:ex}, we will use the two following lemmas.

\begin{lem}\label{lem:1}
For every sufficiently small $\varepsilon > 0$, there exist non-zero complex numbers $a_p$, $p \in \PP$, such that for every $n \geq 0$ and $z_0, \ldots , z_n \in \C$, if $\dist(z_0, \tilde \PP) \geq \varepsilon$ and $|z_k - z_0 - k| < \varepsilon/2$ for $k = 0, \ldots, n$, then 
\[
\sum_{k=0}^n |e(z_k)| < \frac{\varepsilon}{2}.
\]
\end{lem}

\begin{proof}
Take a small $\varepsilon > 0$ and points $z_0, \ldots , z_n$ fulfilling the condition stated in the lemma. Then $z_k = z_0 + k + \zeta_k$, where $\zeta_k \in \C$, $|\zeta_k| < \varepsilon/2 < 1$.

First, consider case (ii), $\PP = \Z$ (the case $\PP=\Z_+$ is analogous). Then
\begin{equation}\label{eq:e}
\sum_{k=0}^n |e(z_k)| \leq \sum_{k=0}^n\sum_{p \in \Z} \frac{|a_p|}{|z_k-p|^2}.
\end{equation}
Note that
\begin{equation}\label{eq:eps1}
|z_k - p| = |z_0 + k + \zeta_k - p| \geq |z_0 + k - p| - |\zeta_k| \geq \varepsilon - \frac{\varepsilon}{2} = \frac{\varepsilon}{2},
\end{equation}
because $|z_0 + k- p| \geq \dist(z_0, \tilde \PP) \geq \varepsilon$. Let
\[
k_0 = [\Re(z_0)].
\]
We will show that
\begin{equation}\label{eq:eps}
|z_k - p| \geq \frac{\varepsilon}{4}|k - p + k_0|.
\end{equation}
Indeed, if $|k - p + k_0| \leq 2$, then \eqref{eq:eps} holds, because by \eqref{eq:eps1},
\[
|z_k - p|\geq \frac{\varepsilon}{2} \geq\frac{\varepsilon}{4}|k - p +k_0|.
\]
Otherwise,
$|k - p + k_0| \geq 3$, so
\begin{align*}
|z_k - p| &= |z_0 + k + \zeta_k - p|\\ &\geq |\Re(z_0 + k + \zeta_k  - p)| \\ 
&= |\Re(z_0) + k + \Re(\zeta_k)  - p|\\ 
&\geq |k - p + k_0| - |\Re(z_0) - k_0| - |\Re(\zeta_k)|\\
&\geq |k - p + k_0| - 2\\ 
&\geq \frac{1}{3}| k - p + k_0|,
\end{align*}
which shows \eqref{eq:eps} if $\varepsilon$ is sufficiently small. 
By \eqref{eq:e}, \eqref{eq:eps1} for $k=p-k_0$ and \eqref{eq:eps} for $k\in \Z\setminus \{p-k_0\}$,
\begin{align*}
\sum_{k=0}^n |e(z_k)| &\leq \sum_{p \in \Z}|a_p|\left(\frac{4}{\varepsilon^2} + \frac{16}{\varepsilon^2} \sum_{k\in\Z \setminus \{p-k_0\}}\frac{1}{|k_0 + k - p|^2}\right)\\
&\leq \sum_{p \in \Z}|a_p|\left(\frac{4}{\varepsilon^2} + \frac{16}{\varepsilon^2} \sum_{k\in\Z \setminus \{0\}}\frac{1}{k^2}\right)< \frac{\varepsilon}{2}\\
\end{align*}
if $|a_p|$ are chosen sufficiently small. 

Now we consider case (iii). Then, writing $p = j+im$ for $j, m \in \Z$, we have
\begin{equation}\label{eq:e''}
\sum_{k=0}^n |e(z_k)| \leq \sum_{k=0}^n\sum_{j \in \Z} \sum_{m \in \Z}\frac{|a_{j+im}|}{|z_k-j-im|^2}.
\end{equation}
Similarly as previously, 
\begin{equation}\label{eq:eps1''}
|z_k - j - im| = |z_0 + k + \zeta_k - j -im| \geq |z_0 + k - j -im| - |\zeta_k| \geq \varepsilon - \frac{\varepsilon}{2} = \frac{\varepsilon}{2},
\end{equation}
since $|z_0 - j +k - im| \geq \dist(z_0, \tilde \PP) \geq \varepsilon$. Set
\[
k_0 = [\Re(z_0)], \qquad m_0 = [\Im(z_0)].
\]
We will show that
\begin{equation}\label{eq:eps''}
|z_k - j -i m| \geq \frac{\varepsilon}{8}(|k+k_0-j| + |m - m_0|).
\end{equation}
To prove \eqref{eq:eps''}, we note that if $|k+k_0-j| + |m - m_0| \leq 4$, then by \eqref{eq:eps1''},
\[
|z_k - j - im|\geq \frac{\varepsilon}{2} \geq\frac{\varepsilon}{8}(|k+k_0-j| + |m - m_0|),
\]
which gives \eqref{eq:eps''}. Otherwise, $|k+k_0-j| + |m - m_0| \geq 5$ and
\begin{align*}
|z_k - j -im| &= |z_0 + k + \zeta_k - j -im|\\ &\geq \frac{\sqrt{2}}{2}(|\Re(z_0 + k + \zeta_k - j - im)| + |\Im(z_0 + k + \zeta_k  - j - im)|)\\
&= \frac{\sqrt{2}}{2}(|\Re(z_0) + k + \Re(\zeta_k) -j| + |\Im(z_0) + \Im(\zeta_k) - m)|)\\
&\geq \frac{\sqrt{2}}{2}(|k+k_0-j| + |m-m_0|- |\Re(z_0) - k_0| - |\Im(z_0) - m_0| - |\Re(\zeta_k)| - |\Im(\zeta_k)|)\\
&\geq \frac{\sqrt{2}}{2}(|k+k_0-j| + |m-m_0|-4)\\
&\geq \frac{\sqrt{2}}{10}(|k+k_0-j| + |m-m_0|),
\end{align*}
which shows \eqref{eq:eps''} for sufficiently small $\varepsilon$. By \eqref{eq:e''}, \eqref{eq:eps1''} and \eqref{eq:eps''},
\begin{align*}
\sum_{k=0}^n |e(z_k)| &\leq  \sum_{j \in \Z} \sum_{m \in \Z} |a_{j+im}|\left( \frac{4}{\varepsilon^2} + \frac{64}{\varepsilon^2}\sum_{k\in\Z \setminus \{j-k_0\}}\frac{1}{(|k+k_0-j| + |m - m_0|)^2}\right)\\
&\leq   \sum_{j \in \Z} 
\sum_{m \in \Z} |a_{j+im}| \left( \frac{4}{\varepsilon^2} + \frac{64}{\varepsilon^2}
\sum_{k\in\Z \setminus \{0\}}\frac{1}{k^2}\right) < \frac{\varepsilon}{2}
\end{align*}
if $|a_{j+im}|$ are sufficiently small.

Finally, case (i) is proved exactly as case (iii) by setting $j=0$.
\end{proof}

For a small $\varepsilon > 0$ let 
\[
V = \C \setminus \bigcup_{p\in\tilde\PP} \D(p, \varepsilon), \qquad
\tilde V = \C \setminus \bigcup_{p\in\tilde\PP} \D(p, 2\varepsilon).
\]

\begin{lem} \label{lem:2}
For every $\varepsilon > 0$, there exist non-zero complex numbers $a_p$, $p \in \PP$, such that for every $z \in \tilde V$ and every $n \geq 0$, 
\[
|f^n(z)-z-n| < \frac{\varepsilon}{2}.
\]
In particular, this implies $f^n(z) \in V$ for every $n \geq 0$.
\end{lem}
\begin{proof}
Take $a_p$, $p \in \PP$ from Lemma~\ref{lem:1} and $z \in \tilde V$. We prove, by induction on $n$, the following claim:
\begin{equation}\label{eq:ind}
\dist(f^n(z), \tilde V) \leq \sum_{k=0}^{n-1}|e(f^k(z))| \qquad \text{and} \qquad |f^n(z)-z-n| < \frac{\varepsilon}{2}
\end{equation}
(with the convention $\sum_{k=0}^{-1}=0$ for $n=0$). 
For $n=0$ the claim \eqref{eq:ind} is trivial. Suppose it holds for $0, \ldots, n$. By the definition of $f$,
\[
f^{n+1}(z) = f^n(z)+1 +e(f^n(z)),
\]
so by induction,
\begin{align*}
\dist(f^{n+1}(z), \tilde V) &\leq \dist(f^n(z) + e(f^n(z)), \tilde V)\\
&\leq \dist(f^n(z), \tilde V) + |e(f^n(z))|\\
&\leq \sum_{k=0}^n|e(f^k(z))|
\end{align*}
and $|f^k(z)-z-k| < \varepsilon/2$ for $k =0, \ldots, n$. Hence, the points $z_k = f^k(z)$, $k =0, \ldots, n$ fulfil the assumptions of Lemma~\ref{lem:1}, so there exist $|a_p|$ small enough so that
\[
\sum_{k=0}^n|e(f^k(z))|< \frac{\varepsilon}{2}.
\]
This together with the definition of $f$ implies
\[
|f^{n+1}(z) - z - n - 1| = |\sum_{k=0}^n e(f^k(z))| \leq \sum_{k=0}^n |e(f^k(z))|< \frac{\varepsilon}{2},
\]
which proves the claim \eqref{eq:ind} for $n+1$, completing the induction. Note that \eqref{eq:ind} implies $f^n(z) \in V$ for $k =0, \ldots, n$, since $z+n \in \tilde V$ and $\dist(\tilde V, \C\setminus V) = \varepsilon$. This ends the proof of the lemma.
\end{proof}

\begin{proof}[Proof of Theorem~{\rm \ref{thm:ex}}] Take a small $\delta > 0$. Set $\varepsilon = \delta/2$ and consider a map of the form \eqref{eq:f} for the numbers $a_p$, $p \in \PP$ satisfying the conditions of Lemmas~\ref{lem:1} and~\ref{lem:2}.  
By Lemma~\ref{lem:2}, we have $f^n \to \infty$ as $n \to \infty$ almost uniformly on $\tilde V$ and $\tilde V$ is connected, so $\tilde V = \C \setminus \bigcup_{p \in \tilde \PP} \D(p, \delta)$ is contained in an invariant Baker domain $U$ for $f$. 

In case (i), we have $1 \in \tilde V \subset U$. By Lemma~\ref{lem:2},
\[
|f^{n+1}(1)- f^n(1)| < 1 + \varepsilon < 2
\]
and
\[
\dist(f^n(1), \bd U) \ge \dist(f^n(1), \bd \tilde V) > n + 1 - \frac 5 2 \varepsilon > n
\]
for $n \geq 0$. Hence, by Lemma~\ref{lem:hyp}, 
\[
\varrho_U(f^{n+1}(1), f^n(1)) \leq 2 \ln \frac{1}{1 - |f^{n+1}(1)- f^n(1)|/\dist(f^n(1), \bd U)} < 2 \ln \frac{1}{1 - 2/n} \to 0
\]
as $n \to \infty$. Therefore, Theorem~A implies that $\varrho_U(f^{n+1}(z), f^n(z)) \to 0$ as $n \to \infty$ for $z \in U$ and $f|_U$ is of parabolic~I type. 
For instance, $W = \{z \in \C: \Re(z) > 1\}$ can be taken to be a nice simply connected absorbing domain in $U$ for $f$.  

Consider now case (ii). Then $ik \in \tilde V \subset U$ for every positive integer $k$, and by Lemma~\ref{lem:2},
\[
1/2 < 1 - \varepsilon < |f^{n+1}(ik)- f^n(ik)| < 1 + \varepsilon < 2
\]
and
\[
k/2 < k - \frac 5 2 \varepsilon < \dist(f^n(ik), \bd \tilde V)  \le \dist(f^n(ik), \bd U)  < k + \frac 5 2 \varepsilon < 2k
\]
for $n \geq 0$. Hence, by Theorem~\ref{thm:rippon} for $K = \{ik, f(ik)\}$, 
\[
\varrho_U(f^{n+1}(ik), f^n(ik)) \ge \frac 1 C \frac{|f^{n+1}(ik) - f^n(ik)|}{\dist(f^n(ik), \bd U)} > \frac{1}{4Ck}> 0,
\]
so Theorem~A implies $\varrho_U(f^{n+1}(z), f^n(z)) \not\to 0$ as $n \to \infty$ for $z \in U$. Moreover, by Lemma~\ref{lem:hyp}, 
\[
\varrho_U(f^{n+1}(ik), f^n(ik)) \leq 2 \ln \frac{1}{1 - |f^{n+1}(ik)- f^n(ik)|/\dist(f^n(ik), \bd U)} < 2 \ln \frac{1}{1 - 4/k} \to 0
\]
as $k \to \infty$, which shows $\inf_{z\in U}\lim_{n\to\infty}\varrho_U(f^{n+1}(z), f^n(z))=0$ (cf.~Remark~\ref{rem:limit}).

Consider now case (iii) and take $z \in U$. Then there exists $p \in\PP$, such that 
\[|z - p| \leq \frac{\sqrt{2}} 2.
\]
Lemma~\ref{lem:2} implies that 
\[
|f(z) - p| < 1 + \frac{\sqrt{2}}2 + \frac{\varepsilon}2 < 2
\]
and 
\[
|f(z)- z| > 1 - \varepsilon > 1/2.
\]
Note that $p, p+1 \notin U$, so by the Schwarz--Pick Lemma,
\[
\varrho_U(f(z), z) \ge \varrho_{\C \setminus \{p, p+1\}}(f(z), z) = \varrho_{\C \setminus \{0, 1\}}(w, w'),
\]
where $w = f(z) - p$, $w' = z - p$ and $w \in \D(0, 2)$, $w' \in \D(0, \sqrt{2}/2)$, $|w-w'| > 1/2$. This easily implies that if $\gamma$ is the hyperbolic geodesic connecting $w$ and $w'$ in $\C \setminus \{0, 1\}$, then there is $\tilde \gamma \subset \gamma$, such that $\tilde \gamma \subset \D(0, 2)$ and the Euclidean length of $\tilde\gamma$ is larger than $1/2$. By Lemma~\ref{lem:CG}, there exists $c > 0$ (independent of $z$), such that $\varrho_{\C \setminus \{0, 1\}}(u) > c$ for every $u \in \D(0, 2)\setminus \{0, 1\}$. Hence, 
\[
\varrho_{\C \setminus \{0, 1\}}(w, w') = \int_\gamma \varrho_{\C \setminus \{0, 1\}}(u)du \ge \int_{\tilde \gamma} \varrho_{\C \setminus \{0, 1\}}(u)du > \frac{c}{2},
\]
which shows
$\inf_{z\in U}\varrho_U(f(z), z) > c/2 > 0$, so $\inf_{z\in U}\lim_{n\to\infty}\varrho_U(f^{n+1}(z), f^n(z))>0$ and, by Proposition~\ref{prop:hyper}, $f|_U$ is of hyperbolic type.

In all three cases (i)--(iii), the boundary of the square 
\[
Q_k = \{z \in \C: |\Re(z)| \leq k +1/2, \; |\Im(z)| \leq k +1/2\}
\]
is contained in $U$ for every integer $k \geq 0$, which shows that $\{\infty\}$ is a singleton component of $\clC \setminus U$. Moreover, in cases~(ii) and~(iii), if $\gamma = \bd Q_0$, then $f^n(\gamma)$ is not contractible in $U$. To see this, notice that by Lemma~\ref{lem:2}, the points of $f^n(\gamma)$ are $\varepsilon/2$-close to the suitable points of the boundary of the square $Q_0 + n$, which winds once around the pole $n$ of $f$. Similarly as in the proof of Theorem~B, this implies that $f^n(\gamma)$ winds once around the pole $n$, so it is not contractible in $U$. By Proposition~\ref{prop:simply}, we conclude that in cases~(ii) and~(iii) there is no simply connected absorbing domain $W$ in $U$ for $f$.

\end{proof}

\bibliography{absorb}

\end{document}